\documentclass[10pt,a4paper]{article}
\usepackage{amsfonts,amsmath,amssymb,amstext,amsbsy,amsopn,amsthm}
\usepackage{mathptmx}
\usepackage{url}
\usepackage{graphicx}
\usepackage{color}

\def\R{{\mathbb R}}
\def\Z{{\mathbb Z}}
\def\C{{\mathbb C}}

\newcommand{\itg}{\int \limits}

\newcommand{\ee}{{\rm e}}
\def\supp{{\rm supp\,}}
\def\im{{\rm Im\,}}
\def\re{{\rm Re\,}}

\DeclareMathOperator{\erfc}{erfc}

\def\a{{\alpha}}

\def\e{\varepsilon}

\def\D{\Delta}
\def\t{\tau}
\def\h{\eta}

\def\d{\delta}
\def\O{\Omega}

\def\k{\kappa}
\def\c{\chi}

\newcommand{\cD}{\mathcal D}

\newcommand{\cO}{\mathcal O}
\newcommand{\cP}{\mathcal P}
\newcommand{\cQ}{\mathcal Q}

\newcommand{\cS}{\mathcal S}

\newcommand{\cE}{\mathcal E}

\def\bx{{\mathbf x}}
\def\by{{\mathbf y}}
\def\bm{{\mathbf m}}
\def\bk{{\mathbf k}}
\def\bP{{\mathbf P}}
\def\bQ{{\mathbf Q}}

\newtheorem{thm}{Theorem}[section]

\newcommand{\keywords}{{\bf Keywords.  }}
\newcommand{\subjclass}{{\bf  Mathematics Subject Classification (2000). }}

\title{Accurate computation of the  high dimensional diffraction potential over
  hyper-rectangles}

\author{ Flavia Lanzara 
\thanks{Department of Mathematics, 
Sapienza University of Rome, 
Piazzale Aldo Moro 2, 00185 Rome, Italy. 
{\it email:} \texttt{\rm lanzara\symbol{'100}mat.uniroma1.it}}
\and Vladimir Maz'ya \thanks{Department of Mathematics, University of
Link\"oping,  581 83 Link\"oping, Sweden.} \thanks{
Department of Mathematical Sciences, University of
Liverpool, Liverpool L69 3BX, UK.
 {\it email:} \texttt{\rm vlmaz\symbol{'100}mai.liu.se }}
\and
Gunther Schmidt
\thanks{
Lichtenberger Str. 12 10178 Berlin, Germany.
{\it email:}schmidt.gunther@online.de}}

\date{}

\begin{document}
\maketitle

\begin{abstract}
  We propose a fast method for high order approximation of  potentials of the Helm\-holtz type operator  \(\Delta+\kappa^2\) over hyper-rectangles in \(\R^n\). By using the basis functions introduced in the theory of approximate approximations, the cubature of a potential is reduced to the quadrature of
  one-dimensional integrals with separable integrands.  Then a separated representation of the density, combined with a suitable quadrature rule, leads to a tensor product representation of the integral operator.  Numerical tests show that these formulas are accurate and provide approximations of order \(6\) up to dimension \(100\) and  \(\kappa^2=100\).

\end{abstract}

\keywords Helmholtz potential; separated representations; higher dimensions.


\pagestyle{empty}

\section{Introduction}
 
We consider the multidimensional diffraction potential
 \begin{equation}\label{volhelm}
 \cS_\k g(\bx)=\itg_{\R^n} \cE_\k (\bx-\by) g(\by)\,d\by \, ,
 \end{equation}
where \(g\) is compactly supported, the wave number \(\k\) is positive  and 
\[
\cE_\k(\bx)= \frac{i}{4} \left(\frac{\k}{2\pi |\bx|}\right)^{n/2-1}H_{n/2-1}^{(1)}(\k|\bx|)
\] 
is the fundamental solution of the Helmholtz type operator \(\Delta+\k^2\). Here  $H_n^{(1)}$ is the $n$-th order Hankel function of the first kind (cf. \cite{AS}). The function $u=\cS_\k g$ is the solution of the Helmholtz equation
\begin{equation*}\label{helmholtz}
\D u+\k^2 u=-g(\bx)
\end{equation*}
satisfying Sommerfeld's radiation condition
\[
\lim_{|\bx|\to\infty}|\bx|^{(n-1)/2} \left(\big< \frac{\bx}{|\bx|}, \nabla u(\bx)\big>-i\, \k u(\bx)\right)=0\,.
\]

The diffraction potential appears frequently in problems of acoustics, electromagnetics and optics. Beside the singularity of the kernel \(\cE_\k\) the approximation of this integral operator is challenging because of the fast oscillations of \(\cE_\k\) for high wave number \(\k\). In this paper we propose a method for high order approximations of \eqref{volhelm}  which is fast and accurate.
We use the concept of approximate approximations introduced by V. Maz'ya in \cite{M1,M2} (see also \cite{MSbook}) and the idea of tensor-structured approximation, first addressed in \cite{BMoh2,BMoh1}. 

Recently  modern methods based on tensor product approximations have been applied to some classes of multidimensional integral operators. They are  based on the approximation of the  kernel by a linear combination of exponentials or Gaussians, which lead to a tensor product approximation. 
For example, in \cite{kho} a separable approximation for the Helmholtz kernel in dimension \(n=3\) was considered. We use a different method, which does not approximate or modify the kernel of the integral operator and provides new very efficient semi-analytic cubature formulas. This approach was already applied to elliptic and parabolic problems.
The construction of approximation formulas for the potential of the operator  $-\Delta +2{\bf b}\cdot \nabla+c$ with ${\bf b} \in\C^n$ and $c\in\C$ over the full space under the condition $\re (c+|{\bf b}|^2)\geq 0$ was considered  in \cite{LMS2}. The harmonic potential over half-spaces was studied  in \cite{LMS3}. Cubature formulas for operators $-\D+c$  over hyper-rectangules in $\R^n$, under the condition  \(\re c\geq 0\), were studied in  \cite{LMS4}.
 The method was extended in \cite{LMS5} to the case \(\re c<0\)  if there exists \(\theta\in\C\) with \(\re \theta > 0\) 
 and $\re\,(\theta \, c)\geq 0$.  
For the diffraction potential with \(c=-\kappa^2<0\) some formulas from
 \cite{LMS5} could result in  numerical overflow problems.
 In this paper we modify  these formulas and 
 show that also the diffraction potential can be treated with our approach.
Note also that
 this approach was applied to parabolic problems in \cite{LMS6} and to higher order operators in \cite{LMS18} and \cite{LMS18bis}.   

The method {\it approximate approximations} uses quasi-interpolation formulas of the type
\begin{equation}\label{quasiint}
g_{h,\cD}(\bx)=\cD^{-n/2} \sum_{\bm \in \Z^n} g(h \bm) \h\left(\frac{\bx-h\bm}{h\sqrt{\cD}}\right).
\end{equation}
Here $h$ and $\cD$ are positive parameters and $\h$ is a smooth and rapidly decaying function. 
Under the assumption that
\begin{equation}\label{moment}
\itg_{\R^n} \bx^\a \h(\bx)d\bx=\d_{0,\a},\qquad 0\leq |\a|<N\, ,
\end{equation}
it was proved in \cite[p.21]{MSbook} that, for any saturation error $\e>0$, one can fix the parameter $\cD>0$ such that
\begin{equation}\label{est}
|g(\bx)-g_{h,\cD}(\bx)|=\cO((h\sqrt{\cD})^N+\e)||g||_{W_\infty^N}\,.
\end{equation}
Using \eqref{quasiint}, we can approximate \(\cS_\k g\) by the sum
\begin{equation}\label{sum}
\cS_\k g_{h,\cD}(\bx)=\cD^{-n/2} \sum_{\bm\in\Z^n} g(h\bm) \Phi(\bx-h\bm)
\end{equation}
with
\[
\Phi(\bx)=(\cS_\k\h(\frac{|\cdot|}{h\sqrt{\cD}}))(\bx)=\itg_{\R^n} \cE_\k\left( \bx-\by\right)\h(\frac{\by}{h\sqrt{\cD}})\,d\by\,,
\]
which gives rise to cubature formulas  with the approximation behavior \eqref{est}. This follows from the boundedness of \(\cS_\k:C(\Omega)\to C(\Omega)\) for any bounded domain \(\Omega\). It remains to choose \(\h\) such that \(\cS_{\k}\h\) can be computed analytically or, at least, efficiently.

As for many other important integral operators of mathematical physics the Gaussian $\ee^{-|\bx|^2}$ and related functions play an important role as generating functions for cubature formulas. For example the function $\pi^{-n/2}\ee^{-|\bx|^2}$ satisfies the moment condition of order $2$ and provides cubature formulas of second order. We construct cubature formulas of order $2M$  by using generating  functions based upon the exponential.
 
The outline of the paper is the following. In Section \ref{sec2} we describe the algorithm for diffraction potentials over the whole space. In Section \ref{sec3}, we derive the one-dimensional integral representations of the diffraction potential over boxes for our basis functions. In Section \ref{sec4}, for densities with separated representations, we describe tensor product approximations of the integral operator and we provide results of numerical experiments showing that these approximations are accurate and preserve the predicted convergence order.

\section{The diffraction potential  over \(\R^n\)}\label{sec2}

Consider the basis functions
\begin{equation*}\label{radialbasis}
\h_{2M}(\bx)=\pi^{-n/2} L_{M-1}^{(n/2)}(|\bx|^2)\ee^{-|\bx|^2}
\end{equation*}
which satisfy the moment conditions \eqref{moment} (cf. \cite[p.55]{MSbook}). Here \(L_j^{(\gamma)}\) are the generalized Laguerre polynomials.
The diffraction potential of \(\h_{2M}\) has the representation
(cf. \cite{Lipp} and \cite[pp. 94-97]{MSbook})
\begin{align*}
\cS_\k \h_{2M}(\bx)=\frac{\cS_\k(\ee^{-|\cdot|^2})(\bx)}{\pi^{n/2}}\sum_{j=0}^{M-1} \frac{\k^{2j}}{4^j j!}
+\frac{\ee^{-|\bx|^2}}{\pi^{n/2}\k^2} \sum_{j=0}^{M-1} \sum_{m=0}^{j-1}
\frac{\k^{2(j-m)}m!}{2^{2(j-m)}j!}L_m^{(n/2-1)}(|\bx|^2)
\end{align*}
with
\[
 \cS_\k(\ee^{-|\cdot|^2})(\bx)=\frac{\pi i \, \ee^{-|\bx|^2}}{2} \left(\frac{\k}{2|\bx|}\right)^{n/2-1}\int_0^\infty H^{(1)}_{n/2-1}(\k r) I_{n/2-1}(2|\bx|r)\ee^{-r^2}r\, dr\,.
\]
 \(I_{n/2-1}\) denotes the modified Bessel functions of the first kind. The function \(\cS_\k \h_{2M}\) is the basis for approximate cubature formulas of order \(O(h^{2M})\). For example, for \(n=3\), we have
\[
\cE_\k(\bx)=\frac{\ee^{i \k |\bx|}}{4\pi|\bx|}
\]
and 
\[
 \cS_\k(\ee^{-|\cdot|^2})(\bx)=\int_{\R^3}\frac{\ee^{i \k |\bx-\by|}}{4\pi|\bx-\by|}\ee^{-|\by|^2}d\by=\frac{\sqrt{\pi}}{2}\frac{\ee^{-|\bx|^2}}{4|\bx|}\left(W(\frac \k 2-i|\bx|)-W(\frac \k 2+i|\bx|)\right)\,.
 \]
The function \(W(z)\) denotes the scaled complementary error function (cf. \cite[7.1.3]{AS})
\begin{equation}\label{fadd1}
W(z)=\ee^{-z^2} \erfc(-iz)=\ee^{-z^2}\frac{2}{\sqrt{\pi}} \int_{-iz}^{\infty}=\frac{2}{\sqrt{\pi}}\int_0^\infty \ee^{-t^2} \ee^{2izt}dt
\end{equation}
which is also known as Faddeeva function. Then the problem is reduced to the efficient computation of certain special functions.

Another approximation of order \(2M\) can be derived with basis function in tensor product form
\begin{equation*}
\begin{split}\label{basis}
&\widetilde\h_{2M}(\bx)=\prod_{j=1}^n{\h}_{2M}(x_j);
\quad
{\h}_{2M}(x)=\frac{(-1)^{M-1}}{2^{2M-1}\sqrt{\pi} (M-1)!}\frac{H_{2M-1}(x) \ee^{-x^2}}{x}\,,
\end{split}
\end{equation*}
where $H_k$ are the Hermite polynomials
\begin{equation*}\label{hermite}
H_k(x)=(-1)^k \ee^{x^2} \left( \frac{d}{dx}\right)^k \ee^{-x^2}.
\end{equation*}
\(\widetilde{\h}_{2M}\) satisfies the moment conditions of order  $2M$  (cf. \cite[p.52]{MSbook}) and the quasi-interpolant \eqref{quasiint}
provides an approximation of $g$ with the error estimate \eqref{est}.

Hence the sum \eqref{sum} with
\[
\Phi_M(\bx)=\itg_{\R^n} \cE_\k\left( \bx-\by\right)\prod_{j=1}^n\h_{2M}(\frac{y_j}{h\sqrt{\cD}})\,d\by
\]
provides a simple cubature formula for the diffraction potential \(\cS_\kappa g\).
\begin{thm}\label{thm1}
If $n\geq 3$, the solution of the equation 
\begin{equation}\label{helm3}
-(\Delta+\k^2)u=\prod_{j=1}^n \h_{2M}(a_j x_j),\qquad a_j>0
\end{equation}
can be expressed by the following one-dimensional integral
\begin{equation}\label{Sk1}
u(\bx)=i \pi^{-n/2} \itg_0^\infty \ee^{i\k^2 t}  \prod_{j=1}^n \cP_M(a_j x_j,4i a_j^2 t) \ee^{-(a_j x_j)^2/(1+4it a_j^2)}dt
\end{equation}
with
\begin{equation}\label{PM}
\cP_M(x,t)=
\sum_{s=0}^{M-1} \frac{(-1)^s}{s! 4^s } \frac{1}{(1+  t)^{s+1/2}} H_{2s}\left(\frac{ x}{\sqrt{1+  t}}\right).
\end{equation}
\end{thm}

\begin{proof}
We multiply \eqref{helm3} by $i$ and look for the solution of
\[
-i \D_\bx\,v -i \,\k^2\,v=i \prod_{j=1}^n \h_{2M}({a_j} x_j),\qquad \bx\in\R^n.
\]
We can obtain $v$ by solving the following Cauchy problem for the parabolic equation in $\R^n$
\begin{equation}\label{eqw}
\begin{split}
&\partial_t \,w-i\, \D_\bx\, w-i\, \k^2\, w=0,\qquad t\geq 0\\
&w(\bx,0)=i\,\prod_{j=1}^n \h_{2M}({a_j} x_j).
\end{split}
\end{equation}
Integrating in $t\in[0,T]$, $T>0$,  we derive 
\[
w(\bx,T)-w(\bx,0)-i(\Delta_\bx+ \k^2) \itg_0^T w(\bx,t)dt=0.
\]
Hence, when $T\to \infty$, we get
\begin{equation}\label{v}
v(\bx)=\itg_0^\infty w(\bx,t)dt
\end{equation}
provided the improper integral exists. If $w$ solves  \eqref{eqw},  then $z=w \, \ee^{-i\k^2 t}$ is the solution of the initial value problem
\[
\partial_t z-i\D_\bx z=0, \qquad z(\bx,0) =i \,\prod_{j=1}^n \h_{2M}({a_j} x_j).
\]
From \cite[(22)]{LMS17} we obtain
\[
z(\bx,t)=i\, \pi^{-n/2}\prod_{j=1}^n \cP_M(a_j x_j,4i a_j^2 t) \ee^{-(a_j x_j)^2/(1+4it a_j^2)}\,.
\]
Since \(w=\ee^{-i\k^2 t} z\),  the assertion follows from \eqref{v}.
\end{proof}

Consequently, we derive  the following approximation formula for  \eqref{volhelm}
\begin{equation}\label{approxS}
\cS_\k g_{h\sqrt{\cD}} (\bx)=\frac{i}{\pi^{n/2} \cD^{n/2}}
\sum_{\bm\in\Z^n}  g(h \bm)
\itg_0^\infty \ee^{i\k^2 t}  \prod_{j=1}^n \cP_M\left(\frac{x_j-hm_j}{h\sqrt{\cD}},\frac{4it}{h^2\cD}\right) \ee^{-\frac{(x_j-hm_j)^2}{h^2{\cD}+4it}}dt.
\end{equation}

The representation \eqref{Sk1} has the advantage that the integrand consists of elementary functions and has separated representation i.e. is the sum of products of one-dimensional functions. This  is useful in multidimensional computations. Indeed, suppose that $g(\bx)$ allows a separated representation that is, within a given accuracy $\e$, it can be represented as sum of products of univariate functions
\begin{equation}\label{sep}
\begin{split}
{g}(\bx)=
\sum_{p=1}^P \alpha_p \prod_{j=1}^n g^{(p)}_j (x_j)+\cO(\e)\, .
\end{split}
\end{equation}
Then a separated representation of \eqref{approxS} can be obtained by applying a quadrature rule. We deduce an approximation formula which is fast also in high dimensional case because only one-dimensional operations are used. This will be treated in detail in section \ref{sec4}. 

It is possible to prove that the cubature error of the described formulas converge to zero as \(h\to 0\).
\begin{thm} \label{thm1.2} (\cite[Theorem 2]{Lipp})
Assume that $g \in C^N(\R^n)$, \(N=2M\), has compact support and
let the mesh width satisfy $h\k  \le C < 2 \pi$ with $\k$
the wavenumber. Then for any $\varepsilon > 0$ there exists  $\cD > 0$ such that
the cubature formulas generated by the functions 
$\eta_{2M}$ or $\widetilde \eta_{2M}$
provide the error estimate 
\[
 |\cS_\k g_{h,\cD}(\bx) - \cS_\k g(\bx)| \le c \, (\sqrt{{\cD}}h)^N
    \sum _{|\alpha|  = N} 
  \frac{\|\partial ^{\alpha} g \|_{L_\infty}}
{\alpha !} +  h^2 \varepsilon \|g\|_{W_\infty^{N-1}} \, .
\]

\end{thm}

 \section{The diffraction potential over hyper -rectangles}\label{sec3}
 We consider integrals
  \begin{equation}\label{volhelm2}
 \cS_\k^{[\bP,\bQ]} g(\bx)=\itg_{[\bP,\bQ]} \cE_\k(\bx-\by) g(\by)\,d\by
 \end{equation}
 taken over rectangular domains  \([\bP,\bQ]=\prod_{j=1}^n[ P_j,Q_j] \).
Then \(\cS_\k^{[\bP,\bQ]} g\) provides a solution of the equation of \((\Delta+\k^2)u=-\chi_{[\bP,\bQ]}g\), together with an appropriate radiation condition.
The direct application of the method described in section \ref{sec2}, which is based on replacing the density by a quasi-interpolant \(g_{h,\cD}\), does not give good approximations because \(g_{h,\cD}\) does not approximate \(g\) near the boundary of \([\bP,\bQ]\). In order to overcome this difficulty,  we extend \(g\) outside \([\bP,\bQ]\) with preserved smoothness and  consider the quasi-interpolant of the extension \(\widetilde g\). Since $\widetilde\h_{2M}$ is smooth and of rapid decay, for any $\e>0$ one can fix $r>0$ and positive parameter $\cD>0$ such that the quasi-interpolant
\begin{equation*}\label{extfbis}
g_{h,\cD}^{(r)}(\bx)= \cD^{-n/2}
\sum_{h\bm\in\O_{rh}}  \widetilde g (h\bm)
\prod_{j=1}^n \eta_{2M} \Big( \frac{x_j - h m_j}{\sqrt{\cD} h}\Big) 
\end{equation*}
with $\O_{rh}=\prod_{j=1}^n I_j$, $I_j= (P_j-rh\sqrt{\cD},Q_j+r h \sqrt{\cD})$, approximates $g$ in $[\bP,\bQ]$ with the error estimate \eqref{est}.
We obtain the following approximation formula of high order for the volume potential \eqref{volhelm2}:
\begin{equation*}\label{cub3}
 \cS_\k^{[\bP,\bQ]}g_{h,\sqrt{\cD}}^{(r)}(\bx) =\cD^{-n/2} \sum_{h\bm\in\O_{rh}}  \widetilde g(h \bm)\Phi^{[\bP-h\bm,\bQ-h\bm]}(\bx-h\bm)\,
\end{equation*}
with
\[
\Phi_M^{[\bP,\bQ]}(\bx)=\itg_{[\bP,\bQ]} \cE_\k\left( \bx-\by\right)\prod_{j=1}^n\h_{2M}(\frac{y_j}{h\sqrt{\cD}})\,d\by\,.
\]
We can prove, similarly to theorem \ref{thm1}, that also \(\Phi_M^{[\bP,\bQ]}\) admits a one-dimensional integral representation where the integrand has separated representation.

\begin{thm}\label{thm2}
Let \(n\geq 3\). The solution of the equation 
\begin{equation}\label{advecM}
-(\D
+\k^2) \, u = \prod_{j=1}^n \c_{(P_j,Q_j)}(x_j) \, 
\eta_{2M}(a_j x_j)\,, \qquad a_j>0
\end{equation}
can be expressed by the one-dimensional integral
\begin{equation*}\label{rep2}
u(\bx)=i \itg_0^\infty \ee^{i \k^2 t}\prod_{j=1}^{n} (\Psi_M( a_j x_j,  4\,i\,a_j^2 t, a_j P_j)-\Psi_M( a_j x_j,4\,  i \, a_j^2 t,  a_j Q_j))dt\,,
\end{equation*}
where 
\begin{equation*}\label{Psi}
\Psi_M(x,t,y)=\frac{1}{2\sqrt{\pi}}
\ee^{- x^2/(1+t)}
 \Bigg(\erfc(F(x,t,y)) \cP_M(x,t) 
-  \frac{\ee^{-F^2(x,t,y)}}{\sqrt{\pi} } \cQ_M(x,t,y)
\Bigg)
\end{equation*}
with $\cP_M$ defined in \eqref{PM}, 
\[
F(x,t,p)=\sqrt{\frac{t+1}{t}} \left(p-\frac{x}{t+1}\right),\qquad
\cQ_1(x,t,y)=0,
\]
\begin{align}\nonumber
\cQ_M(x,t,y) &=2
\sum_{k=1}^{M-1} \frac{(-1)^k}   {k! \, 4^k}
  \sum_{\ell=1}^{2k}
\frac{(-1)^{\ell}}{t^{\ell/2}}
\bigg( H_{2k-\ell} (y)
  H_{\ell-1} \Big( \frac{y-x}{\sqrt{t}}\Big) \\\nonumber  
&\hskip60pt -
\Big(\hskip-4pt\begin{array}{c}2k\\\ell\end{array}\hskip-4pt\Big)
H_{2k-\ell}\Big(\frac{x}{\sqrt{1+t}}\Big)
\frac{H_{\ell-1}\big(F(t, x, y)\big)}{(1+t)^{k+1/2}}
\bigg)\,, \> M>1.
\end{align}
\end{thm}

\begin{proof}
We multiply \eqref{advecM} by $i$ and look for the solution of
\begin{equation*}\label{advecMe}
-i\D\,u-ik^2 \, u = i\prod_{j=1}^n \c_{(P_j,Q_j)}(x_j) 
\eta_{2M}(a_j x_j)\, .
\end{equation*}
As in theorem \ref{thm1}, we have
\[
u(\bx)=\itg_0^\infty \ee^{i k^2 t}v(\bx,t)dt\,,
\]
where \(v\) solves the Cauchy problem for the parabolic equation 
\[
\partial_t v-i\, \Delta_\bx v=0,\qquad v(\bx,0)=i \prod_{j=1}^n  \c_{(P_j,Q_j)}(x_j) \,
\eta_{2M}(a_j x_j)\, .
\]
The assertion follows from \cite[theorem 4.1]{LMS17}.
\end{proof}
At the grid points \(\{h\bk\}\) we obtain the cubature formula
\begin{equation*}
\cS_\k^{[\bP,\bQ]}g(h\bk)\approx \cD^{-n/2}\sum_{h\bm\in\O_{rh}}  \widetilde g(h \bm)b_{\bk,\bm}^{(M)}\,,
\end{equation*}
where we introduce the one-dimensional integrals
\[
b_{\bk,\bm}^{(M)}=
i \itg_0^\infty \ee^{i k^2 t}\prod_{j=1}^{n} (\Psi_M(\frac{k_j-m_j}{\sqrt{\cD}} ,  \frac{4\,i\, t}{h^2\cD}, \frac{P_j-hm_j}{h\sqrt{\cD}} )-\Psi_M( \frac{k_j-m_j}{\sqrt{\cD}}, \frac{4\,  i \, t}{h^2\cD},  \frac{Q_j-hm_j}{h\sqrt{\cD}} ))dt\,.
\]

\section{Separated representation and numerical results}\label{sec4}

The problem is reduced to find efficient quadrature formulas for the integrals  \(b_{\bk,\bm}^{(M)}\). It is well known that the trapezoidal rule for rapidly decaying integrands converges fast. If we make the substitutions
\begin{equation}\label{tm}
t=\ee^\xi,\quad \xi=a(\t+\ee^\t)\quad {\rm and} \quad \tau=b (u-\ee^{-u})
\end{equation}
with certain positive constants \(a, b\), as proposed in \cite{Mo}, then the integrals \(b_{\bk,\bm}^{(M)}\) transform to integrals over \(\R\) with integrands decaying doubly exponentially. We get after the substitutions \eqref{tm} 
\begin{multline*}
b_{\bk,\bm}^{(M)}=
i \itg_0^\infty \ee^{i k^2 \Phi(u)}\prod_{j=1}^{n} (\Psi_M(\frac{k_j-m_j}{\sqrt{\cD}} ,  \frac{4\,i\, \Phi(u)}{h^2\cD}, \frac{P_j-hm_j}{h\sqrt{\cD}} )\\
-\Psi_M( \frac{k_j-m_j}{\sqrt{\cD}}, \frac{4\,  i \, \Phi(u)}{h^2\cD},  \frac{Q_j-hm_j}{h\sqrt{\cD}} ))\Phi'(u)\,du\, ,
\end{multline*}
 where we set
 \(
 \Phi(u)=\exp(a b(u-\exp(-u))+a\,\exp(b(u-\exp(-u))))
 \).
 The quadrature with the trapezoidal rule with step size \(\tau\) gives a separated representation for \(b_{\bk,\bm}^{(M)}\).
 Assuming a separated representation \eqref{sep} of the density \(\widetilde g\) we derive the approximation of \eqref{volhelm} using one-dimensional operations
 \begin{multline}\label{finapp}
 \frac{i\t}{\cD^{n/2}}\sum_{p=1}^P \alpha_p \sum_{s=-N_0}^{N_1}\ee^{i \k^2 \Phi(s\t)} \Phi'(s\t) \prod_{j=1}^n \sum_{m_j}\widetilde g_j^{(p)}(hm_j)\times\\
  (\Psi_M(\frac{k_j-m_j}{h\sqrt{\cD}}, \frac{4\,i\, \Phi(s\t)}{h^2\cD}, \frac{P_j-hm_j}{h\sqrt{\cD}} )-\Psi_M( \frac{k_j-m_j}{h\sqrt{\cD}}, \frac{4\,  i \, \Phi(s\t)}{h^2\cD},  \frac{Q_j-hm_j}{h\sqrt{\cD}}))\,.
 \end{multline}
The problem is reduced to an efficient  implementation of \(\Psi_M\). We use the representation of the complementary error function \(\erfc (z)=\ee^{-z^2}W(iz)\) via the Faddeeva function \eqref{fadd1} and write
\begin{align*}
\Psi_M(x,t,y)
&=\frac{\ee^{-y^2}\ee^{i(y-x)^2/t}}{2\sqrt{\pi}}
\left( W\big(i F(x,it,y)\big)\cP_M(x,it) -\frac{\cQ_M(x,it,y)}{\sqrt{\pi} } 
\right)\,.
\end{align*}
Efficient implementations of  double precision computations of $W(z)$ 
 are available
if the imaginary part of the argument is nonnegative.
Otherwise, for $\im z < 0$,
overflow problems can occur,
which can be seen from the relation $W(z)=2\ee^{-z^2}-W(-z)$ (cf. \cite[7.1.11]{AS}).
But this helps to derive a stable formula also for $\im(i F(x,it,y))=\re F(x,it,y) < 0$,
since
\begin{align*}
\frac{\ee^{-x^2/(1+it)}\ee^{-F^2(x,it,y)}}{2} &W(iF(x,it,y))
 =\frac{\ee^{-x^2/(1+it)}\ee^{-F^2(x,it,y)}}{2}\Big(2 \ee^{F^2(x,it,y)}-W\big(-i F(x,it,y)\big)\Big)\\
&=\ee^{-x^2/(1+it)}-\frac{\ee^{-y^2}\ee^{i(y-x)^2/t}}{2}W\big(-i F(x,it,y)\big) \, .
\end{align*}
Thus we get the efficient formula
\begin{equation*}\label{PsiNew}
\begin{split}
&\Psi_M(x,t,y)= \, -\frac{\ee^{-y^2+i(y-x)^2/t}}{2\sqrt{\pi}}
\frac{\cQ_M(x,it,y)}{\sqrt{\pi}} +\\
&\left\{
\begin{aligned}
& \ee^{-y^2+i(y-x)^2/t} \, 
W\big(i F(x,it,y)\big)\frac{\cP_M(x,it)}{2\sqrt{\pi}}  , &\re F(x,it,y) \geq 0 ,\\
&\Big(2\,{\ee^{-x^2/(1+it)}}-{\ee^{-y^2+i(y-x)^2/t}} \, W\big(-i F(x,it,y)\big)\Big)\frac{\cP_M(x,it)}{2\sqrt{\pi}},
& \re F(x,it,y)<0 .
\end{aligned}
\right.
\end{split}
\end{equation*}
We verify numerically the accuracy of formula \eqref{finapp} and the convergence order of the method.
We assume in \eqref{volhelm}
\[
g(\bx)=-(\Delta+\k^2)\prod_{j=1}^n w(x_j),\quad \bx\in[-1,1]^n ,
\]
with \(\supp w\subset [-1,1]\) and \(w(\pm 1)=w'(\pm 1)=0\) which has the exact value \(\cS_\k^{[-{\bf 1},\bf{1}]} g=\prod_{j=1}^n w(x_j)\). 
In  the next tables we consider 
\[
w(x)=(x^2-1)^2 \ee^{x} \quad {\rm for} \quad  x\in[-1,1]; \quad w(x)=0\quad {\rm otherwise}\,,
\]
the extension \(\widetilde w(x)=w(x)\) (other extensions can also be considered, see e.g. \cite{LMS17}),   \(\cD=3\) and the parameters in the quadrature rule \(\tau=10^{-6}\),  \(a=6\), \(b=4\) in \eqref{tm}. 

In Table \ref{T1} we report on exact values and absolute errors for
\(\cS_\k^{[-{\bf 1},\bf{1}]} g\) at some points \((x,0,...,0)\), where
we have chosen the space dimensions \(n=3,10,100\) and the wave number
\(\k^2=1, 10,100\).  In \eqref{finapp} we choose \(h=0.025\) and
\(M=3\).  In Tables \ref{T2},\ref{T3} and \ref{T4} we report on the
absolute errors and the approximation rates for the diffraction
potential \(\cS_\k^{[-{\bf 1},\bf{1}]} g(0.2,...,0)\) in the space
dimensions \(n=3\) (Table \ref{T2}), \(n=10\) (Table \ref{T3}),
\(n=100\) (Table \ref{T4}) and the wave numbers \(\k^2=1,10,100\).
The approximate values are computed by the cubature formula
\eqref{finapp} for \(M=1,2,3\) and different values of the step size
\(h\).
The results show that in dimension \(n=3\), for large \(\kappa^2\),
the sixth order formula fails probably due to the slow decay of the
integrand and its rapid oscillations, whereas in dimensions \(n=10,
100\) formula \eqref{finapp} approximates \(\cS_\k^{[\bP,\bQ]}g\)
accurately, with the predicted approximation rate also for large
\(\kappa^2\).

\begin{table}
\begin{footnotesize}
\begin{center}
\begin{tabular}{r|c|c|c|c} \hline
\multicolumn{5}{c}{\(n=3\)}\\ \hline
$x$ & exact value & error  ({\(\k^2=1\)}) & error  ({\(\k^2=10\)}) &  error  ({\(\k^2=100\)}) \\[1pt] \hline       
 -0.4&  0.4729778245  & 0.1570150E-07 &  0.3615417E-04&0.1213741E-02 \\   
    0&1.0000000000  &0.1545729E-07& 0.3617374E-04 & 0.1213832E-02\\
   0.4& 1.0526315067 &  0.1430657E-07&0.3618822E-04&0.1213810E-02 \\
     0.8& 0.2884301043 & 0.1626331E-07&0.3618271E-04&0.1213701E-02 \\
         1.2& 0.0000000000 & 0.1347441E-07&0.3616111E-04 & 0.1213396E-02\\
\hline 
\end{tabular}\\[1mm]

\begin{tabular}{r|c|c|c|c} \hline
\multicolumn{5}{c}{\(n=10\)}\\ \hline
$x$ & exact value & error  ({\(\k^2=1\)}) & error  ({\(\k^2=10\)}) &  error  ({\(\k^2=100\)}) \\[1pt] \hline       
 -0.4&  0.4729778245 & 0.2220996E-07&0.1819134E-07&0.1238302E-06 \\   
    0&1.0000000000 & 0.4879996E-07&0.4217366E-07&0.2686369E-06  \\
   0.4&1.0526315067  &0.5801273E-07&0.5039037E-07 & 0.1988678E-06\\
     0.8&  0.2884301043  &0.3110463E-07&0.2394059E-07&0.1591014E-06  \\
         1.2& 0.0000000000  &0.4434357E-08&0.4529327E-08& 0.5568088E-07\\
\hline 
\end{tabular}\\[1mm]

\begin{tabular}{r|c|c|c|c} \hline
\multicolumn{5}{c}{\(n=100\)}\\ \hline
$x$ & exact value & error  ({\(\k^2=1\)}) & error  ({\(\k^2=10\)}) &  error  ({\(\k^2=100\)}) \\[1pt] \hline       
 -0.4& 0.4729778245 & 0.2423595E-06 &0.2423583E-06&0.2423444E-06 \\   
    0&1.0000000000&0.5133366E-06&0.5133336E-06& 0.5133042E-06 \\
   0.4&1.0526315067&  0.5480781E-06&0.5480750E-06&0.5480403E-06 \\
     0.8& 0.2884301043 & 0.1760679E-06&0.1760238E-06& 0.1753438E-06 \\
         1.2& 0.0000000000&  0.7227814E-09&0.7543741E-09 &  0.1230072E-08 \\
\hline 
\end{tabular}\\[1mm]

\caption{\small Exact value and absolute error
for $\cS_\k^{[{\bf 1},{\bf 1}]} g(x,0,...,0)$ using \eqref{finapp} with  \(\cD=3\), \(h=0.025\), \(M=3\).}\label{T1} 
\end{center}
\end{footnotesize}
\end{table}

\begin{table}
\begin{footnotesize}
\begin{center}
\begin{tabular}{r|cc|cc|cc} \hline
\(n=3\)&\multicolumn{6}{c}{\(\k^2=1\)}\\ \hline
 & \multicolumn{2}{c|}{\(M=1\)}& \multicolumn{2}{c|}{\(M=2\)}& \multicolumn{2}{c}{\(M=3\)}   \\\hline
$h^{-1}$ &  abs. error & rate  &abs. error & rate& abs. error & rate  \\[1pt] \hline       
 5&0.182E+01  &  & 0.198E+00   &   &0.752E-02  &\\
10&   0.403E+00  &    2.18      &   0.131E-01&   3.92   &0.112E-03&6.07\\
 20& 0.991E-01 &     2.02      & 0.814E-03 & 4.01    & 0.148E-05 &  6.24  \\
40 & 0.247E-01  &    2.00      &   0.506E-04 &4.01    & 0.147E-07  &    6.66 \\
80 & 0.617E-02&    2.00   &  0.319E-05 & 3.99   &0.276E-07 &   \\
\hline 
\end{tabular}\\[1mm]

\begin{tabular}{r|cc|cc|cc} \hline
&\multicolumn{6}{c}{\(\k^2=10\)}\\ \hline
 & \multicolumn{2}{c|}{\(M=1\)}& \multicolumn{2}{c|}{\(M=2\)}& \multicolumn{2}{c}{\(M=3\)}   \\\hline
$h^{-1}$ &  abs. error & rate  &abs. error & rate& abs. error & rate  \\[1pt] \hline       
  5&  0.127E+01 &  &0.461E+00  &   & 0.457E-01&\\
 10& 0.261E+00  & 2.28     & 0.293E-01 &  3.97    & 0.652E-03&    6.13\\
 20& 0.633E-01 &    2.05  &  0.179E-02 &  4.03
   & 0.131E-04&5.64   \\
40 &  0.157E-01 & 2.01   &  0.102E-03 & 4.13  & 0.466E-04&  \\
80 &   0.387E-02 &    2.02 &  0.383E-04 &  1.42 &  0.392E-04 &  \\
\hline 
\end{tabular}\\[1mm]
\begin{tabular}{r|cc|cc|cc} \hline
&\multicolumn{6}{c}{\(\k^2=100\)}\\ \hline
 & \multicolumn{2}{c|}{\(M=1\)}& \multicolumn{2}{c|}{\(M=2\)}& \multicolumn{2}{c}{\(M=3\)}   \\\hline
$h^{-1}$ &  abs. error & rate  &abs. error & rate& abs. error & rate  \\[1pt] \hline       
 5&  0.512E+00 &  & 0.291E+00   &   & 0.201E-01 &\\
10& 0.149E+00 &    1.78 &  0.173E-01  &     4.07   & 0.922E-03  & 4.45\\
 20& 0.399E-01 & 1.90   & 0.436E-03 &  5.31    &0.130E-02 &   \\
40 & 0.974E-02 &  2.04   &  0.119E-02  &   &0.121E-02   &   \\
80 &  0.391E-02 &   1.32 &  0.154E-02&   &0.155E-02  &    \\
\hline 
\end{tabular}\\[1mm]
\caption{\small Absolute errors and rate of convergence
for $\cS_\k g(0.2,0,0)$ using \eqref{finapp}.}\label{T2} 
\end{center}
\end{footnotesize}
\end{table} 

\begin{table}
\begin{footnotesize}
\begin{center}
\begin{tabular}{r|cc|cc|cc} \hline
\(n=10\)&\multicolumn{6}{c}{\(\k^2=1\)}\\ \hline
 & \multicolumn{2}{c|}{\(M=1\)}& \multicolumn{2}{c|}{\(M=2\)}& \multicolumn{2}{c}{\(M=3\)}   \\\hline
$h^{-1}$ &  abs. error & rate  &abs. error & rate& abs. error & rate  \\[1pt] \hline       
10& 0.256E+00 &  
  &    0.515E-03 & 
  & 0.583E-03 &  
  \\
 20&0.693E-01 &1.88 &  0.887E-05& 5.86 & 0.885E-05  &  6.04  \\
40 &  0.177E-01  &1.97  &    0.204E-06 & 5.44   &0.135E-06 &    6.03   \\
80 & 0.444E-02 &1.99&   0.107E-07 &  4.25 & 0.397E-08 & 5.09   \\
\hline 
\end{tabular}\\[1mm]

\begin{tabular}{r|cc|cc|cc} \hline
&\multicolumn{6}{c}{\(\k^2=10\)}\\ \hline
 & \multicolumn{2}{c|}{\(M=1\)}& \multicolumn{2}{c|}{\(M=2\)}& \multicolumn{2}{c}{\(M=3\)}   \\\hline
$h^{-1}$ &  abs. error & rate  &abs. error & rate& abs. error & rate  \\[1pt] \hline      
10&   0.301E+00& 
&   0.264E-02 & 
 & 0.583E-03&    
 \\
 20& 0.813E-01 &1.89 & 0.181E-03 &3.86    & 0.135E-06 &  6.03  \\
40 &0.207E-01 & 1.97& 0.116E-04& 3.97&  0.501E-07&  6.03  \\
80 &  0.521E-02&1.99   & 0.723E-06 &4.00   &0.397E-08  &   5.09   \\
\end{tabular}\\[1mm]
\begin{tabular}{r|cc|cc|cc} \hline
&\multicolumn{6}{c}{\(\k^2=100\)}\\ \hline
 & \multicolumn{2}{c|}{\(M=1\)}& \multicolumn{2}{c|}{\(M=2\)}& \multicolumn{2}{c}{\(M=3\)}   \\\hline
$h^{-1}$ &  abs. error & rate  &abs. error & rate& abs. error & rate  \\[1pt] \hline       
10& 0.811E+00  && 0.719E-01& 
&   0.255E-02 & 
\\
 20&  0.256E+00 & 1.67   & 0.440E-02& 4.03  &0.391E-04&  6.03  \\
40 &  0.675E-01&1.92   & 0.273E-03& 4.01  &   0.605E-06& 6.01   \\
80 &  0.171E-01&1.98   & 0.171E-04 &4.00  & 0.705E-08 &6.42   \\
\hline 
\end{tabular}\\[1mm]
\caption{\small Absolute errors and rate of convergence
for $\cS_\k g(0.2,0,...,0)$ using \eqref{finapp}.}\label{T3} 
\end{center}
\end{footnotesize}
\end{table}

\begin{table}[h]
\begin{footnotesize}
\begin{center}
\begin{tabular}{r|cc|cc|cc} \hline
\(n=100\)&\multicolumn{6}{c}{\(\k^2=1\)}\\ \hline
 & \multicolumn{2}{c|}{\(M=1\)}& \multicolumn{2}{c|}{\(M=2\)}& \multicolumn{2}{c}{\(M=3\)}   \\\hline
$h^{-1}$ &  abs. error & rate  &abs. error & rate& abs. error & rate  \\[1pt] \hline       
10&  0.101E+01 && 0.864E-02 &   
&  0.591E-02 &   
\\
 20& 0.487E+00  & 1.05   & 0.314E-03 &4.78    &  0.895E-04&  6.05  \\
40 &0.148E+00&  1.71 & 0.161E-04&4.28  &  0.136E-05 & 6.04 \\
80 & 0.391E-01 &   1.93& 0.989E-06 &4.02&   0.422E-07&  5.01\\
\end{tabular}\\[1mm]

\begin{tabular}{r|cc|cc|cc} \hline
&\multicolumn{6}{c}{\(\k^2=10\)}\\ \hline
 & \multicolumn{2}{c|}{\(M=1\)}& \multicolumn{2}{c|}{\(M=2\)}& \multicolumn{2}{c}{\(M=3\)}   \\\hline
$h^{-1}$ &  abs. error & rate  &abs. error & rate& abs. error & rate  \\[1pt] \hline       
10& 0.101E+01 &  &   0.864E-02  &     
& 0.591E-02&  
\\
 20& 0.487E+00 & 1.05  & 0.314E-03 & 4.78    & 0.895E-04&  6.05  \\
40 &0.148E+00 &  1.71   &  0.161E-04&    4.28   &  0.136E-05&  6.04   \\
80 &0.391E-01&  1.93   &  0.989E-06 &    4.02   & 0.422E-07  &  5.01  \\
\hline 
\end{tabular}\\[1mm]
\begin{tabular}{r|cc|cc|cc} \hline
&\multicolumn{6}{c}{\(\k^2=100\)}\\ \hline
 & \multicolumn{2}{c|}{\(M=1\)}& \multicolumn{2}{c|}{\(M=2\)}& \multicolumn{2}{c}{\(M=3\)}   \\\hline
$h^{-1}$ &  abs. error & rate  &abs. error & rate& abs. error & rate  \\[1pt] \hline       
10&0.101E+01  & &0.864E-02 & 
&  0.591E-02 &   
\\
 20&0.487E+00 &    1.05   &    0.314E-03& 4.78 & 0.895E-04& 6.05
 \\
40 & 0.148E+00 &  1.71  &   0.161E-04  &4.28   &0.136E-05&   6.04 \\
80 &   0.391E-01& 1.93 &  0.989E-06 & 4.02  & 0.422E-07&  5.01    \\
\hline 
\end{tabular}\\[1mm]
\caption{\small Absolute errors and rate of convergence
for $\cS_\k g(0.2,0,...,0)$ using \eqref{finapp}.}\label{T4} 
\end{center}
\end{footnotesize}
\end{table}

\end{document}